\documentclass{article}
\usepackage{amsmath,amsfonts,amsthm,amssymb,amscd}

\newcommand{\esssup}{\mathop{\rm ess\ sup}}

\newcommand{\p}{\partial}
\newcommand{\e}{\varepsilon}

\newcommand{\R}{{\mathbb R}}

\newcommand{\Z}{{\mathbb Z}}

\newcommand{\T}{{\mathbb T}}

\newcommand{\aA}{{\cal A}}

\newcommand{\DD}{{\cal D}}

\newcommand{\FF}{{\cal F}}

\newcommand{\KK}{{\cal K}}

\newcommand{\RR}{{\cal R}}
\newcommand{\sS}{{\cal S}}

\newcommand{\XX}{{\cal X}}

\newcommand{\diver}{\mathop{\rm div}\nolimits}

\theoremstyle{plain}
\newtheorem{theorem}{Theorem}[section]
\newtheorem*{mt}{Main Theorem}
\newtheorem{lemma}[theorem]{Lemma}
\newtheorem{proposition}[theorem]{Proposition}

\theoremstyle{definition}
\newtheorem{definition}[theorem]{Definition}

\theoremstyle{remark}

\numberwithin{equation}{section}

\begin{document}
\author{Armen Shirikyan} 
\title{Euler equations are not exactly controllable\\
by a finite-dimensional external force}
\date{\small CNRS (UMR 8088), D\'epartement de Math\'ematiques\\
Universit\'e de Cergy--Pontoise, 
Site de Saint-Martin\\ 2 avenue Adolphe Chauvin\\
95302 Cergy--Pontoise Cedex, France\\ 
E-mail: Armen.Shirikyan@u-cergy.fr}

\maketitle
\begin{abstract}
We show that the Euler system is not exactly controllable by a
finite-dimensional external force. The proof is based on the comparison of 
the Kolmogorov $\e$-entropy for H\"older spaces and for the class of
functions that can be obtained by solving the 2D Euler equations with
various right-hand sides.


\end{abstract}
\tableofcontents
\setcounter{section}{-1}

\section{Introduction}
\label{s0}
Let us consider the controlled Euler system on the 2D torus~$\T^2$:
\begin{equation} 
\dot u+\langle u,\nabla\rangle u+\nabla p=\eta(t,x),\quad \diver u=0. \label{0.1}
\end{equation}
Here $u$ and $p$ are unknown velocity field and pressure, and~$\eta$ stands for a control force taking values in a finite-dimensional space~$E\subset L^2(\T^2,\R^2)$. Equations~\eqref{0.1} are supplemented with the initial condition
\begin{equation} \label{0.2}
u(0,x)=u_0(x).
\end{equation}
It was proved by Agrachev and Sarychev~\cite{AS-2006} that Eqs.~\eqref{0.1} are approximately controllable in~$L^2$ and exactly controllable in observed projections. More precisely, they constructed a six-dimensional subspace $E\subset C^\infty(\T^2,\R^2)$ such that the following properties hold for any~$T>0$:
\begin{description}
 \item[\it Approximate controllability:]
For any divergence-free vector fields~$u_0$ and~$\hat u$ that belong to the Sobolev space~$H^2(\T^2,\R^2)$ and any~$\e>0$ there is a smooth $E$-valued control~$\eta(t)$ such that the solution~$u$ of problem~\eqref{0.1}, \eqref{0.2} satisfies the inequality
$\|u(T)-\hat u\|_{L^2}<\e$.

 \item[\it Exact controllability in projections:]
For any subspace $F\subset H^2(\T^2,\R^2)$ of  finite dimension, any divergence-free vector field~$u_0\in H^2(\T^2,\R^2)$, and any function~$\hat u\in F$ there is a smooth $E$-valued control~$\eta(t)$ such that $\mathsf P_Fu(T)=\hat u$, where~$\mathsf P_F$ denotes the orthogonal projection in~$L^2$ onto the space~$F$. 
\end{description}
In view of the above results, an important question arises here: is it possible to prove the exact controllability for~\eqref{0.1}, or more generally, given an initial state~$u_0$ and a control space~$E$, what is the set of attainability at a time~$T$, i.~e., the family of functions~$\aA_T(u_0,E)$ that can be obtained at the time~$T$ by solving problem~\eqref{0.1}, \eqref{0.2}? Since the Euler system is time-reversible, a natural class of final states~$\hat u$ for which one may wish to prove the exact controllability is dictated by the regularity of the initial state~$u_0$ and the control~$\eta$. Namely, let us denote by~$C^s$ the H\"older space of order~$s$ on the torus and by~$C_\sigma^s$ the space of divergence-free vector fields $u\in C^s$; see Notations below for the exact definition. Assume that the initial state~$u_0$ and the control~$\eta$ are $C^s$-smooth with respect to the space variables. In this case, if $s>1$, then the solution~$u(t)$ belongs to~$C^s$ for any $t\ge0$. Conversely, for any divergence-free vector field $\hat u\in C^s$ we can find $u_0\in C^s$ such that the solution of~\eqref{0.1}, \eqref{0.2} with $\eta\equiv0$ issued from~$u_0$ coincides with~$\hat u$ at $t=T$. Thus, it is reasonable to study the problem of exact controllability for the class of final states that are as regular as the initial function and the control. The following theorem, which is a simplified version of the main result of this paper, shows that the set of attainability is much smaller than the above-mentioned class of functions.

\begin{mt}
Let $u_0$ be an arbitrary divergence-free vector field belonging to the H\"older space $C^s$ with a non-integer~$s>2$ and let $E\subset C^s$ be any finite-dimensional subspace. Then, for any $T>0$, the complement in~$C_\sigma^s$ of the set of attainability $\aA_T(u_0,E)$ is everywhere dense in~$C_\sigma^s$. 
\end{mt}

The proof of this theorem is based on two key observations. The first of them is the Lipschitz continuity of the resolving operator for~\eqref{0.1}, \eqref{0.2} with respect to the controls~$\eta$ endowed with the relaxation norm\footnote{The relaxation norm of~$\eta$ is defined as the least upper bound of the norm for the integral of~$\eta$ with respect to  time.} (see Theorem~6 in~\cite{AS-2006} and Proposition~\ref{t1.3} below). It is curious that this property is also crucial for proving the approximate controllability and exact controllability in projections~\cite{AS-2006}. The second key ingredient of the proof is an upper bound for the $\e$-entropy of the space of controls. Roughly speaking, we combine these two properties to establish an upper bound for the $\e$-entropy for set of attainability~$\aA_T(u_0,E)$ with given initial function $u_0\in C_\sigma^s$ and control space $E\subset C^s$. It turns out that this upper bound is much smaller than the $\e$-entropy of~$C_\sigma^s$, and the required property follows. 

It should be mentioned that the above theorem is false in the case when~$E$ is the space of functions supported by a given domain $D\subset\T^2$. In this situation, it is well known that the Euler system is exactly controllable (see~\cite{coron-1996} and~\cite{glass-2000} for the 2D and 3D cases, respectively).  

\smallskip
In conclusion, let us note that the Kolmogorov $\e$-entropy has proved to be an effective tool for studying various problems in analysis. For instance, we refer the reader to~\cite{Mitjagin-1961, Lorentz-1966, Lorentz1986, KH1995,  VC-1998, CE-1999, Zelik-2001, CV2002} for a number of applications of the $\e$-entropy in approximation theory, dynamical systems, and theory of attractors. This paper shows that it can also be used in the control theory for PDE's. 

\medskip
{\bf Acknowledgements}. I am grateful to P.~G\'erard 
for discussion on the Euler equations. 

\bigskip
\subsection*{Notations}
Let $X$ be a Banach space with a norm~$\|\cdot\|_X$, let $J\subset\R$ be a finite closed interval, let~$s>0$ be a non-integer, and let~$\T^d$ be the $d$-dimensional torus. We shall use the following function spaces.

\medskip
\noindent
$L^p(J,X)$ is the space of Bochner-measurable functions $f:J\to X$ such that
$$
\|f\|_{L^p(J,X)}:=\biggl(\int_J\|f(t)\|_X^pdt\biggr)^{1/p}<\infty.
$$
In the case $p=\infty$, the above norm should be replaced by
$$
\|f\|_{L^\infty(J,X)}:=\esssup_{t\in J}\|f(t)\|_X.
$$

\smallskip
\noindent
$W^{1,p}(J,X)$ stands the space of functions $f\in L^p(J,X)$ such that $\p_t f\in L^p(J,X)$. It is endowed with the natural norm. In the case $X=\R$, we shall write $L^p(J)$ and $W^{1,p}(J)$. 

\smallskip
\noindent
$C(\T^d)$ is the space of continuous functions $u:\T^d\to\R^d$ with the norm
$$
\|u\|:=\sup_{x\in\T^d}|u(x)|.
$$

\smallskip
\noindent
$C^s(\T^d)$ is the H\"older class of order~$s$ with the norm
$$
\|u\|_s:=\max_{|\alpha|\le[s]}\|\p^\alpha u\|
+\max_{|\alpha|=[s]}\,\sup_{x\ne y}\frac{|\p^\alpha u(x)-\p^\alpha u(y)|}{|x-y|^\gamma},
$$
where $\p^\alpha$ is a standard notation for derivatives, $[s]$ stands for the integer part of~$s$, and $\gamma=s-[s]$.

\smallskip
\noindent
$C_\sigma^s(\T^d)$ denotes the space of functions $u\in C^s(\T^d)$ such that $\diver u\equiv 0$. In the case $d=2$, we shall drop~$\T^d$ from the notation and write~$C^s$ and~$C_\sigma^s$. 

\smallskip
\noindent
We denote by $\langle a, b\rangle$ or $a\cdot b$ the usual scalar product  of the vectors $a,b\in\R^2$ and by~$C_1, C_2,\dots$ unessential positive constants. 

\section{Cauchy problem for Euler equations on the 2D torus}
\label{s1}
\subsection{Existence and uniqueness of solution}
Consider the Cauchy problem for the following Euler type system
on the 2D torus~$\T^2$: 
\begin{align} 
\dot u+\langle u+z,\nabla\rangle (u+z)+\nabla p&=f(t,x),
\quad \diver u=0, \label{1.1}\\
u(0,x)&=u_0(x),
\label{1.2}
\end{align}
where $z$, $f$, and $u_0$ are given functions, and $\nabla=(\p_1,\p_2)$. Let us recall the
concept of strong solution for \eqref{1.1}, \eqref{1.2}. We fix a
time interval $J=[0,T]$ and a non-integer $s>1$ and introduce the spaces 
$$
\DD_T:=C_\sigma^s\times W^{1,1}(J,C_\sigma^s)\times L^1(J,C^s),\quad \XX_T:=L^\infty(J,C^s)\cap W^{1,\infty}(J,C^{s-1}),
$$
where the spaces $C^s$, $C_\sigma^s$ and $W^{1,p}$ are defined in the Introduction (see Notations). The spaces~$\DD_T$ and~$\XX_T$ are endowed with natural norms.

\begin{definition} \label{d1.1}
Let $(u_0,z,f)\in\DD_T$ be an arbitrary triple.  A pair of functions
$(u,p)$ is called a {\it strong solution of the Cauchy problem for the
Euler type system~\eqref{1.1}\/} if~$u$ and~$p$ belong to the
spaces~$\XX_T$ and $L^1(J,C^s)$, respectively, and Eqs.~\eqref{1.1},
\eqref{1.2} are satisfied in the sense of distributions.
\end{definition}

In what follows, when dealing with solutions of Eq.~\eqref{1.1}, we
shall sometimes omit the function~$p(t,x)$ and write
simply~$u(t,x)$. This will not lead to a confusion because~$p$ can be
found, up to an additive function depending only on time, from the
relation 
$$
\Delta p=\diver\bigl(f-\langle u+z,\nabla\rangle (u+z)\bigr),
$$
which is obtained by taking the divergence of the first equation
in~\eqref{1.1}. 

The following existence and uniqueness result is essentially due to
Wolibner~\cite{Wolibner-1933} and Kato~\cite{Kato-1967} (see
also~\cite{Gerard-1992} for a concise presentation of the proofs).
 
\begin{theorem} \label{t1.2}
For any non-integer $s>1$, any time interval $J=[0,T]$ and an
arbitrary triple $(u_0,z,f)\in\DD_T$, problem~\eqref{1.1}, \eqref{1.2}
has a unique solution $u\in\XX_T$. Moreover, the resolving operator
$$
\RR:\DD_T\to \XX_T,\quad (u_0,z,f)\mapsto u(t,x),
$$ 
is bounded, that is, it maps bounded sets in~$\DD_T$ to bounded sets
in~$\XX_T$.
\end{theorem}

\begin{proof}
In the case $z\equiv0$, existence and uniqueness of solutions
for~\eqref{1.1}, \eqref{1.2} is proved in~\cite{Kato-1967}. The
general case can be reduced to the former by the change of unknown
function $u=v-z$. Boundedness of the resolving operator follows easily
from the proof of existence given in~\cite{Kato-1967}.
\end{proof}

% In what follows, to simplify the presentation, we shall assume that
% the mean value in $x\in\T^2$ of all functions entering~\eqref{1.1} is
% zero. It is well known that this assumption does not restrict the
% generality. 

\subsection{Lipschitz continuity of the resolving operator}
We now study continuity properties of the operator~$\RR$ constructed in 
Theorem~\ref{t1.2}.  Let~$B_{\DD_T}(R)$ be the closed ball in~$\DD_T$ of radius~$R$ centred at origin. The
following proposition is one of the two key points in the proof of
non-controllability for the Euler equations. A similar result in the case of $L^2$-norm in the target space was established earlier by Agrachev and Sarychev~\cite{AS-2006}.

\begin{proposition} \label{t1.3}
For any positive constants~$T$ and~$R$ and any non-integer $s>2$ there
is $C=C(T,R,s)>0$ such that
\begin{multline} \label{1.3}
\|\RR(u_{01},z_1,f_1)-\RR(u_{02},z_2,f_2)\|_{L^\infty(J,C^{s-1})}\\
\le C\,\bigl(\|u_{01}-u_{02}\|_{C^{s-1}}+\|z_1-z_2\|_{L^1(J,C^s)}
+\|f_1-f_2\|_{L^1(J,C^{s-1})}\bigr),
\end{multline}
where $(u_{0i},z_i,f_i)$, $i=1,2$, are arbitrary triples belonging to
the ball~$B_{\DD_T}(R)$. 
\end{proposition} 

\begin{proof}
Derivation of~\eqref{1.3} is based on a well-known idea of reduction
of the 2D Euler system to a nonlinear transport equation for the vorticity (e.g.,
see~\cite{Gerard-1992}). For the reader's convenience, we give a
detailed proof of the proposition. 
We shall confine ourselves to derivation of~\eqref{1.3} for smooth solutions. The proof in the general case can be carried out by a standard approximation argument.

Let $u(t,x)$ be a smooth solution for~\eqref{1.1}, \eqref{1.2}. Applying the operator $\nabla^\bot=(-\p_2,\p_1)$ to the first relation in~\eqref{1.1} and to~\eqref{1.2}, we obtain
$$
\dot v+\langle u+z,\nabla\rangle(v+\zeta)=g,\quad v(0,x)=v_0(x),
$$
where $v=\nabla^\bot\cdot u$, $\zeta=\nabla^\bot\cdot z$, $g=\nabla^\bot\cdot f$, and $v_0=\nabla^\bot \cdot u_0$. It follows that if $u_i$, $i=1,2$, are two smooth solutions associated with data $(u_{0i},z_i,f_i)$, then the function $v=\nabla^\bot (u_1-u_2)$ is a solution of the problem 
\begin{align}
\dot v+\langle u_2+z_2,\nabla\rangle v&=g-\langle u_2+z_2,\nabla\rangle \zeta-\langle u+z,\nabla\rangle(v_1+\zeta_1), \label{1.4}\\
v(0,x)&=v_0(x),\label{1.5}
\end{align}
where $u=u_1-u_2$,  $z=z_1-z_2$,  $\zeta=\nabla^\bot\cdot z$, $\zeta_i=\nabla^\bot\cdot z_i$, $g=\nabla^\bot\cdot(f_1-f_2)$, and $v_0=\nabla^\bot\cdot(u_{01}-u_{02})$. Thus, $v$~is a solution of an inhomogeneous transport equation associated with the divergence-free vector field $u_2+z_2$. It follows that 
\begin{equation} \label{1.6}
v(t,x)=v_0(U_{0,t}(x))+\int_0^th(\tau,U_{\tau,t}(x))\,d\tau,
\end{equation} 
where $U_{t,\tau}(x)$ denotes the flow defined by the vector field $u_2+z_2$, and~$h$ stands for the right-hand side in~\eqref{1.4}. Let us denote by~$\Delta^{-1}$ the inverse of the Laplace operator in the space of functions on~$\T^2$ with zero mean value. Recalling that
the functions~$u$ and~$v$ are connected by the relations $v=\nabla^\bot\cdot u$ and $u=Gv$, where $G=\nabla^\bot\Delta^{-1}$, from~\eqref{1.6} we derive
\begin{equation} \label{1.60}
u(t,x)=G\bigl[(\nabla^\bot u_0)(U_{0,t}(x))\bigr]
+\int_0^tG\bigl[h(\tau,U_{\tau,t}(x))\bigr]\,d\tau.
\end{equation} 
Now note that $U_{t,\tau}(x)$, $t,\tau\in J$, are diffeomorphisms of the torus with uniformly bounded $C^s$-norms, and the function~$h$ can be written as 
$$
h=\nabla^\bot\cdot f-\diver\bigl(\zeta(u_2+z_2)-(v_1+\zeta_1)(u+z)\bigr),
$$
where $f=f_1-f_2$. Since the operator~$G:C^{s-1}\to C^s$ is bounded (see~\cite[Section~4.3]{GT2001}, taking the $C^{s-1}$-norm of both sides in~\eqref{1.60}, we see that
\begin{equation} \label{1.7}
\|u(t)\|_{s-1}\le C_1\|u_0\|_{s-1}
+C_1\int_0^t\bigl(\|f\|_{s-1}+\|\zeta(u_2+z_2)-(v_1+\zeta_1)(u+z)\|_{s-1}\bigr)\,d\tau,
\end{equation}
where $C_1>0$ depends only on~$R$. The second term under the integral in~\eqref{1.7} can be estimated by
$$
\|\zeta\|_{s-1}\|u_2+z_2\|_{s-1}+\|v_1+\zeta_1\|_{s-1}\|u+z\|_{s-1}
\le C_2\bigl(\|z\|_s+\|u\|_{s-1}\bigr).
$$
Substituting this expression into~\eqref{1.7}, we obtain
$$
\|u(t)\|_{s-1}\le C_1\|u_0\|_{s-1}
+C_3\int_0^t\bigl(\|f\|_{s-1}+\|z\|_s+\|u\|_{s-1}\bigr)\,d\tau,
$$
where $C_3$ is a constant depending only on~$R$. Application of the Gronwall inequality gives the required estimate~\eqref{1.3}. 
\end{proof}

\section{Kolmogorov $\boldsymbol\e$-entropy}
\label{s2}

\subsection{Definition and an elementary property}
Let $X$ be a Banach space and let $K\subset X$ be a compact
subset. Let us recall the concept of $\e$-entropy, which characterises
the ``massiveness'' of~$K$ (e.g., see~\cite{Lorentz1986}). For any
$\e>0$, we denote by $N_\e(K)$ the minimal number of sets of
diameters~$\le2\e$ that are needed to cover~$K$. The {\it Kolmogorov
$\e$-entropy\/} (or simply $\e$-entropy) of~$K$ is defined as
$H_\e(K)=\ln N_\e(K)$. Thus, the $\e$-entropy of a compact set
$K\subset X$ is a non-increasing function of $\e>0$, and it is easy to
see that~$H_\e(K)$ depends only on the metric on~$K$ (and not on the
ambient space~$X$). If we wish to emphasise that~$K$ is endowed with the norm of~$X$, then we shall write $H_\e(K,X)$. 

Now let $Y$ be another Banach space and let $f\!:K\to Y$ be a
Lipschitz-continuous function:
\begin{equation} \label{2.1}
\|f(u_1)-f(u_2)\|_Y\le L\|u_1-u_2\|_X\quad\mbox{for $u_1,u_2\in K$}, 
\end{equation}
where $L>0$ is a constant. The following lemma is a straightforward
consequence of the definition.

\begin{lemma} \label{l2.1}
For any compact set $K\subset X$ and any function $f\!:K\to Y$
satisfying inequality~\eqref{2.1}, we have
\begin{equation} \label{2.2}
H_\e(f(K))\le H_{\e/L}(K)\quad\mbox{for all $\e>0$}.
\end{equation}
\end{lemma}

\subsection{Estimates for the $\e$-entropy of some compact sets}
Let $\varphi_1$ and $\varphi_2$ be two non-increasing functions of
$\e>0$. We shall write $\varphi_1\prec\varphi_2$ if there are positive
constants~$C$ and~$\e_0$ such that
$$
\varphi_1(\e)\le C\varphi_2(\e)\quad\mbox{for $0<\e\le\e_0$}.
$$
If $\varphi_1\prec\varphi_2$ and $\varphi_2\prec\varphi_1$, then we write $\varphi_1\sim\varphi_2$. The second key ingredient of the proof of non-controllability for the Euler system is given by the following two propositions.

\begin{proposition} \label{p2.2}
Let $r<s$ be positive non-integers such that $s-r\notin\Z$ and let
$B\subset C_\sigma^s(\T^d)$ be an arbitrary closed ball. Then, for any $\delta>0$, we have
\begin{equation}  \label{2.3}
H_\e(B,C^r(\T^d))\succ\Bigl(\frac1\e\Bigr)^{\frac{d}{s-r}-\delta}. 
\end{equation}
\end{proposition}

\begin{proof}
Let us recall that if~$Q$ is a closed ball in $C_\sigma^{q}(\T^d)$ with a non-integer $q>0$, then
\begin{equation}  \label{2.4}
H_\e(Q,C(\T^d))\sim\Bigl(\frac1\e\Bigr)^{\frac{d}{q}}; 
\end{equation}
see~\cite[Section~10.2]{Lorentz1986}. Since $C^\nu(\T^d)$ is continuously embedded in~$C(\T^d)$ for any $\nu>0$, it follows from~\eqref{2.4} that if~$A\subset C_\sigma^{s-r+\nu}(\T^d)$ is any closed subset with non-empty interior and $s-r+\nu\notin\Z$, then
\begin{equation}  \label{2.5}
H_\e(A,C^\nu(\T^d))\succ\Bigl(\frac1\e\Bigr)^{\frac{d}{s-r+\nu}}. 
\end{equation}
Furthermore, if $\nu\notin\Z$, then the operator
$(1-\Delta)^{-(r-\nu)/2}$ (where~$\Delta$ is the Laplacian) defines an isomorphism from~$C^\nu(\T^d)$ to~$C^r(\T^d)$ and from~$C_\sigma^{s-r+\nu}(\T^d)$ to~$C_\sigma^s(\T^d)$ (see~\cite[Section~4.3]{GT2001}). Combining this with relation~\eqref{2.5} and Lemma~\ref{l2.1}, we see that 
\begin{equation}  \label{2.55}
H_\e(B,C^r(\T^d))\succ\Bigl(\frac1\e\Bigr)^{\frac{d}{s-r+\nu}}, 
\end{equation}
where $B$ is an arbitrary closed ball in $C_\sigma^s(\T^d)$. It remains to note that the left-hand side of~\eqref{2.55} does not depend on the parameter~$\nu>0$, which can be chosen arbitrarily small.
\end{proof}

\begin{proposition} \label{p2.3}
Let $J=[0,T]$ and let $E$ be a
finite-dimensional vector space. Then, for any closed ball $B\subset W^{1,1}(J,E)$, we have
\begin{equation}  \label{2.6}
H_\e(B,L^1(J,E))\prec\frac1\e\ln\frac1\e. 
\end{equation}
\end{proposition}

\begin{proof}
We first note that it suffices to prove~\eqref{2.6} for scalar functions. Indeed, if~$E$ is an $n$-dimensional vector space, then~$B$ is a subset of the direct product of~$n$ balls $B_1\subset W^{1,1}(J)$. If~\eqref{2.6} is established in the case $\dim E=1$, then
$$
H_{n\e}(B,W^{1,1}(J,E))\le n H_\e(B_1,W^{1,1}(J))
\le\frac{Cn}{\e}\ln\frac1\e; 
$$
see inequality~(7) in Section 10.1 of~\cite{Lorentz1986}.
Replacing~$n\e$ by~$\e$ in the above estimate, we obtain~\eqref{2.6}. 

We now prove~\eqref{2.6} for scalar functions. Without loss of generality, we can assume that $J=[0,1]$ and $B\subset W^{1,1}(J)$ is a closed ball of radius~$R$ centred at zero. Let us fix $\e>0$ and describe a finite family of functions $\FF\subset W^{1,1}(J)$ that form an $\e$-net for~$B$. To this end, we choose sufficiently large integers~$L$ and~$M$ and denote by~$I_k$ the interval $[t_{k-1},t_{k})$, where $t_k=k/L$. The family~$\FF$ consists of all functions $f\in L^1(J)$ that are constant on every interval~$I_k$, $k=1,\dots,L$, and take one of the values $2jR/M$, $j=-M,\dots,M$, on each interval of constancy. It is clear that~$\FF$ consists of $N(L,M):=(2M+1)^L$ elements. Let us show that, for an appropriate choice of~$L$ and~$M$, the family~$\FF$ is an $\e$-net for~$B$. 

We first note that 
$$ %\begin{equation} \label{2.11}
\|u\|_{L^\infty(J)}\le 2R\quad\mbox{for any $u\in B$}.
$$ %\end{equation}
Furthermore, 
$$
|u(t)-u(t_{k-1})|\le \int_{t_{k-1}}^t|\dot u(\tau)|\,d\tau
\quad\mbox{for $t\in I_k$},
$$
whence it follows that 
\begin{align}
\sum_{k=1}^L\int_{I_k}|u(t)-u(t_{k-1})|\,dt
&\le\sum_{k=1}^L\int_{I_k}\int_{t_{k-1}}^t|\dot u(\tau)|\,d\tau\,dt\notag\\
&\le\sum_{k=1}^L\int_{I_k}|\dot u(\tau)|(t_k-\tau)\,d\tau \notag\\
&\le L^{-1}\|\dot u\|_{L^1(J)}\le RL^{-1}. 
\label{2.7}
\end{align}
Now note that for any $L$-tuple $(u_0,\dots,u_{L-1})$ there is $f\in\FF$ such that
\begin{equation} \label{2.8}
|f(t)-u_{k-1}|\le 2RM^{-1}\quad\mbox{for $t\in I_k$, $k=1,\dots,L$}.
\end{equation} 
Combining inequalities~\eqref{2.7} and~\eqref{2.8}, in which $u_k=u(t_k)$, we obtain
\begin{align}
\int_0^1|u(t)-f(t)|\,dt
&=\sum_{k=1}^L\int_{I_k}|u(t)-f(t)|\,dt\notag\\
&\le \sum_{k=1}^L\int_{I_k}\bigl(|u(t)-u(t_{k-1})|+|u(t_{k-1})-f(t)|\bigr)\,dt\notag\\
&\le RL^{-1}+2RLM^{-1}. 
\label{2.9}
\end{align}
Let us set
\begin{equation} \label{2.10}
L=\bigl[2R/\e\bigr]+1,\quad M=\bigl[4RL/\e\bigr]+1,
\end{equation} 
where $[a]$ stands for the integer part of~$a\ge0$. In this case, it follows from~\eqref{2.9} that
$$
\|u-f\|_{L^1(J)}\le E.
$$
Thus, the family $\FF$ is an $\e$-net for~$B$. 

Let us estimate the number of elements in~$\FF$. Relations~\eqref{2.10} imply that 
$$
N(L,M)=(2M+1)^L\le\bigl(C_1\e^{-2}\bigr)^{C_2\e^{-1}}
\le\exp\bigl(C_3\e^{-1}\ln\e^{-1}\bigr).
$$
Taking the logarithm, we arrive at the required estimate~\eqref{2.6}.
\end{proof}

\section{Main result}
\label{s3}

\subsection{Formulation}
\label{s3.1}
Let us fix a time interval $J=[0,T]$ and consider the controlled 2D
Euler system on the domain $J\times\T^2$:
\begin{align} 
\dot u+\langle u,\nabla\rangle u+\nabla p&=h(t,x)+\eta(t,x),
\quad \diver u=0,\label{3.1}\\
u(0,x)&=u_0(x). \label{3.2}
\end{align}
Here~$h$ and~$u_0$ are given functions, and~$\eta$ is a control. In what follows, we fix a non-integer $s>2$ and assume that $h\in L^1(J,C^s)$ and $u_0\in C^s$. Let $E\subset C^s$ be a closed subspace and let $\KK\subset C_\sigma^s$ be any subset. 

\begin{definition} \label{d3.1}
We shall say that the 2D Euler system with given external force $h\in L^1(J,C^s)$ and initial function $u_0\in C_\sigma^s$ is {\it exactly controllable in time~$T$ for the class~$\KK$} if for any $\hat u\in\KK$ there is $\eta\in L^1(J,E)$ such that 
$$
u(T,x)=\hat u(x),
$$
where $u\in\XX_T$ stands for the solution of~\eqref{3.1}, \eqref{3.2}.
\end{definition}

Let us give an equivalent definition of exact controllability in terms of the set of attainability. Let us denote  by~$\RR_t(u_0,f)$ the operator that takes the pair $(u_0,f)\in C_\sigma^s\times L^1(J,C^s)$ to $u(t)\in C^s$, where $u\in\XX_T$ stands for the solution of problem~\eqref{1.1}, \eqref{1.2} with $z\equiv0$. For given~$u_0$ and~$h$, let $\aA_T(u_0,h,E)$ be the image of $L^1(J,E)$ under the mapping $\RR_T(u_0,h+\cdot)$. It is clear that the Euler system is exactly controllable 
in time~$T$ for a class~$\KK\subset C_\sigma^s$ if and only if $\aA_T(u_0,h,E)\supset\KK$. 

Let  $\aA_T^c(u_0,h,E)$ be the complement of $\aA_T(u_0,h,E)$ in the space~$C_\sigma^s$. The following theorem is the main result of this paper. 

\begin{theorem} \label{t3.2}
Let $s>2$ be any non-integer, let $E\subset C^s$ be an arbitrary finite-dimensional subspace, and let $u_0\in C_\sigma^s$ and $h\in L^1(J,C^s)$ be given functions. Then, for any non-negative $\gamma<1$ and any ball $Q\subset C_\sigma^{s+\gamma}$, we have
\begin{equation} \label{3.3}
\aA_T^c(u_0,h,E)\cap Q\ne\varnothing.
\end{equation}
In particular, the 2D Euler system is not exactly controllable in any time~$T$ for the class $C_\sigma^{s+\gamma}$. 
\end{theorem}

\subsection{Proof of Theorem~\ref{t3.2}}
{\it Step~1}. 
We first show that it suffices to consider the case $E\subset C_\sigma^s$. Indeed, let us denote by~$\Pi$ the Leray projection, that is, 
$$
\Pi u=u-\nabla\bigl(\Delta^{-1}(\diver u)\bigr).
$$
The above relation and the continuity of $\Delta^{-1}$ in H\"older spaces (see~\cite{GT2001}) imply that~$\Pi$ is a continuous operator from~$C^s$ to~$C_\sigma^s$. It is well known that 
$$
\RR_T(u_0,f)=\RR_T(u_0,\Pi f),
$$
whence it follows that $\aA_T(u_0,h,E)=\aA_T(u_0,h,\Pi E)$. Thus, if relation~\eqref{3.3} is established for any finite-dimensional subspace $E\subset C_\sigma^s$, then it remains true in the general case. 

\medskip
{\it Step~2}. 
We now assume that $E$ is a finite-dimensional subspace in~$C_\sigma^s$.  Let us write solutions of~\eqref{3.1}, \eqref{3.2} in the form
\begin{equation} \label{3.4}
u(t,x)=v(t,x)+z(t,x), \quad z(t,x)=\int_0^t \eta(\tau,x)\,d\tau.
\end{equation}
In this case, the function~$v$ belongs to the space~$\XX_T$ and satisfies the equations
\begin{equation} \label{3.5}
\dot v+\langle v+z,\nabla\rangle (v+z)+\nabla p=h(t,x),
\quad \diver v=0, \quad
v(0,x)=u_0(x).
\end{equation}
In view of Theorem~\ref{t1.2}, for any $z\in W^{1,1}(J,E)$, problem~\eqref{3.5} has a unique solution $u\in\XX_T$. Let us denote by $\sS:W^{1,1}(J,E)\to\XX_T$ the operator that takes~$z$ to~$u$ and by~$\sS_T$ its restriction to the time~$T$. It follows from~\eqref{3.4} that we can write the solution of~\eqref{3.1}, \eqref{3.2} at the time $t=T$ in the form
\begin{equation} \label{3.6}
\RR_T(u_0,h+\eta)=z(T)+\sS_T(z),
\end{equation}
where $z$ is given by the second relation in~\eqref{3.4}. 

To prove~\eqref{3.3}, we argue by contradiction. Suppose that $\aA_T(u_0,h,E)$ contains a closed ball $Q\subset C_\sigma^{s+\gamma}$. In this case, it follows from~\eqref{3.6} that the image of the space $E\times W^{1,1}(J,E)$ under the mapping $K(y,z):=y+\sS_T(z)$ contains~$Q$. Let us write
\begin{equation} \label{3.7}
E\times W^{1,1}(J,E)=\bigcup_{n=1}^\infty B_n,
\end{equation}
where $B_n$ denotes the closed ball in $E\times W^{1,1}(J,E)$ of radius~$n$ centred at zero. Since the union of~$K(B_n)$ covers~$Q$, by the Baire theorem, there is an integer $m\ge1$ such that~$K(B_m)$ is dense in a ball~$\widehat Q\subset Q$ with respect to the metric of~$C^{s+\gamma}$. Furthermore, Proposition~\ref{t1.3} implies that the mapping~$K$ is continuous from~$E\times L^1(J,E)$ to~$C^{s-1}$. Now note~$B_m$ is compact in~$E\times L^1(J,E)$. It follows that~$K(B_m)$ is closed in~$C^{s-1}$ and, hence, $K(B_m)\cap C^{s+\gamma}$ is closed in~$C^{s+\gamma}$. Thus, $K(B_m)$ contains~$\widehat Q$. On the other hand, we shall show in the next step that 
\begin{equation} \label{3.8}
H_\e(K(B_m),C^{s-1})\prec \e^\nu H_\e(\widehat Q,C^{s-1}),
\end{equation}
where $\nu>0$. 
This contradicts the inclusion $\widehat Q\subset K(B_m)$. 

\medskip
{\it Step~3}. 
Without loss of generality, we can assume that $s+\gamma\notin\Z$. By Proposition~\ref{p2.2}, for any $\delta>0$, we have
$$
H_\e(\widehat Q,C^{s-1})\succ\Bigl(\frac1\e\Bigr)^{\frac{2}{1+\gamma}-\delta}.
$$
Let us choose $\delta>0$ so small that the exponent in the right-hand side of the above relation is bigger than~$1$. Thus, we can find $\alpha>1$ such that
\begin{equation} \label{3.9}
H_\e(\widehat Q,C^{s-1})\succ\Bigl(\frac1\e\Bigr)^\alpha. 
\end{equation}
On the other hand, let us endow~$B_m$  with the metric of $E\times L^1(J,E)$. Since~$E$ is finite-dimensional, for any ball $B\subset E$, we have (see Theorem~10.2 in~\cite{Lorentz1986})
$$
H_\e(B,E)\sim\ln\frac1\e.
$$
Combining this with Proposition~\ref{p2.3}, we see that 
\begin{equation} \label{3.10}
H_\e(B_m,E\times L^1(J,E))\prec\frac1\e\ln\frac1\e.
\end{equation}
It follows from Proposition~\ref{t1.3} that the mapping~$K$ is Lipschitz-continuous from~$B_m$ to~$C^{s-1}$. Relation~\eqref{3.10} and Lemma~\ref{l2.1} now imply that 
\begin{equation} \label{3.11}
H_\e(K(B_m),C^{s-1})\prec\frac1\e\ln\frac1\e.
\end{equation}
The required estimate~\eqref{3.8} is a consequence of~\eqref{3.9} and~\eqref{3.11}. The proof of the theorem is complete. 

\addcontentsline{toc}{section}{Bibliography}

\def\cprime{$'$} \def\cprime{$'$}
\providecommand{\bysame}{\leavevmode\hbox to3em{\hrulefill}\thinspace}
\providecommand{\MR}{\relax\ifhmode\unskip\space\fi MR }
\providecommand{\MRhref}[2]{%
  \href{http://www.ams.org/mathscinet-getitem?mr=#1}{#2}
}
\providecommand{\href}[2]{#2}

\end{document}